\documentclass[a4paper, 12pt, reqno]{amsart}
\usepackage{amsmath}
\usepackage{amssymb}
\usepackage{graphics}
\usepackage{mathrsfs}
\usepackage[pdflatex]{graphicx}
\usepackage[marginratio=1:1,tmargin=117pt, height=650pt]{geometry}
\usepackage{color}
\usepackage{amsthm}

\usepackage[all]{xy}

\usepackage{eurosym}

\usepackage[english]{babel}
\usepackage[T1]{fontenc}
\usepackage[latin1]{inputenc}
\usepackage{marginnote}

\usepackage{amsmath,amsfonts,amsthm,amssymb,verbatim,enumerate,quotes,graphicx}
\usepackage[flushmargin]{footmisc}
\usepackage[noadjust]{cite}
\usepackage{color}
\newcommand{\comments}[1]{}
\numberwithin{equation}{section}
\makeatletter
\def\blfootnote{\xdef\@thefnmark{}\@footnotetext}
\makeatother

\definecolor{orange}{rgb}{1,0.5,0}

\theoremstyle{plain}
\newtheorem{theorem}{Theorem}[section]

\newtheorem{cor}[theorem]{Corollary}

\newtheorem{lemma}[theorem]{Lemma}

\theoremstyle{definition}

\newtheorem{ex}{Example}[section]

\theoremstyle{remark}

\begin{document}

\title[Fast escaping points: a new regularity condition]{Fast escaping points of entire functions:\\
 a new regularity condition}
\author{V. Evdoridou}
\address{Department of Mathematics and Statistics\\ The Open University\\ Walton Hall\\ Milton Keynes MK7 6AA\\ United Kingdom}
\email{vasiliki.evdoridou@open.ac.uk}
\date{\today}
\begin{abstract}
Let $f$ be a transcendental entire function. The fast escaping set, $A(f)$, plays a key role in transcendental dynamics. The quite fast escaping set, $Q(f)$, defined by an apparently weaker condition is equal to $A(f)$ under certain conditions. Here we introduce $Q_2(f)$ defined by what appears to be an even weaker condition.  Using a new regularity condition we show that functions of finite order and positive lower order satisfy $Q_2(f)=A(f).$ We also show that the finite composition of such functions satisfies $Q_2(f)=A(f).$   Finally, we construct a function for which $Q_2(f) \neq Q(f)= A(f).$
\end{abstract}

\maketitle

\section{Introduction}
Let $f$ be a transcendental entire function and denote by $f^n, n=0,1,...,$ the $n$th iterate of $f$. An introduction to the theory of iteration of transcendental entire and meromorphic  functions can be found in \cite{Berg}.\\

The set $$I(f)= \{z \in \mathbb{C}:f^n(z) \to \infty\}$$ is called the \textit{escaping set} and was first studied for a general transcendental entire function by Eremenko in \cite{Ere}, where he conjectured that all the components of $I(f)$ are  unbounded. Although much progress has been made towards the conjecture, it still remains an open problem.\\

Results on Eremenko's conjecture for a general transcendental entire function have been obtained by Rippon and Stallard in \cite{Fast}, \cite{F-E} by considering a subset of the escaping set known as the \textit{fast escaping set}, $A(f),$ and showing that all the components of $A(f)$ are unbounded. This set was introduced by Bergweiler and Hinkkanen in \cite{B-H}. We will use the definition given in \cite{Fast}   according to which
$$A(f)= \{z: \;\text{there exists}\;\ell \in \mathbb{N}\;\text{such that}\; \lvert f^{n+\ell}(z) \rvert \geq M^n(R,f),\;\text{for}\;n \in \mathbb{N}\},$$
where $$M(r)= M(r,f) = \max_{\lvert z\rvert =r} \lvert f(z)\rvert, \;\;\text{for}\;\;r>0,$$ and $R>0$ is large enough to ensure that $M(r) >r$ for $r \geq R.$\\

The set $A(f)$ now plays a key role in complex dynamics (see \cite{Fast}) and so it is useful to be able to identify points that are fast escaping. In \cite[Theorem 2.7]{Fast},\linebreak it is shown that  points which eventually escape faster than the iterates of any function of the form $r \mapsto \varepsilon M(r), r>0$, where $\varepsilon \in (0,1)$,  are actually fast escaping.\\

It is natural to ask whether  $\mu_{\varepsilon}$ can be replaced in this result by a smaller function. In this context,
 Rippon and Stallard introduced the \textit{quite fast escaping set} $Q(f)$ in \cite{Regul}. Let $\mu_{\varepsilon}(r)= M(r)^{\varepsilon},$ where $r>0$ and $\varepsilon \in (0,1).$ The quite fast escaping set is defined as follows: 
 $$Q(f)= \{z:\exists\; \varepsilon \in (0,1)\;\text{and}\;\ell \in \mathbb{N} \;\text{such that}\;\lvert f^{n+\ell}(z)\rvert \geq \mu_{\varepsilon}^n(R),\;\text{for}\; n \in \mathbb{N} \},$$
 where $R>0$ is such that $\mu_{\varepsilon}(r)>r$ for $r \geq R.$ The  function $\mu_{\varepsilon}$ that is used in the definition of $Q(f)$ is smaller than the function defined by $r \mapsto \varepsilon M(r).$\\

 The set $Q(f)$ arises naturally in complex dynamics and so it is of interest to establish when  $Q(f)$ is equal to $A(f).$ Although Rippon and Stallard were the first to define $Q(f)$, points that belong to $Q(f)$ were used earlier in results concerning the Hausdorff measure and Hausdorff dimension of the escaping set and the Julia set of some classes of functions (see \cite{B-K-S} and \cite{Peter}). Rippon and Stallard showed that   $Q(f)=A(f)$ for many classes of functions, but they also constructed examples where $Q(f) \neq A(f).$ One well studied class of functions for which $Q(f)=A(f)$  is the Eremenko-Lyubich class $\mathcal{B}$ which consists of the functions for which the set of singularities of the inverse function, $f^{-1}$, is bounded (see \cite{E-L}).\\
 
 The following family of functions $\mu_{m,\varepsilon}$ is a natural generalisation of the function $\mu_{\varepsilon}$ defined by $\mu_{\varepsilon}(r)=M(r)^{\varepsilon}$:
$$\log^m \mu_{m,\varepsilon}(r)= \varepsilon \log^m M(r),\;\;m \in \mathbb{N},\;\;\varepsilon \in (0,1),$$
whenever $\mu_{m,\varepsilon}(r)$ is defined.
Note that the function $\mu_{\varepsilon}$ used in the definition of $Q(f)$ is equal to $\mu_{1,\varepsilon}$.\\

In this paper we focus on the case $m=2$, that is, we consider 
 \begin{equation}
 \label{defmu}
\mu_{2,\varepsilon}(r)= \exp((\log M(r))^{\varepsilon}),\;\;0<\varepsilon<1,
\end{equation}
 and we set
 \begin{equation}
\label{q2def} 
  Q_2(f)= \{z: \exists\; \varepsilon \in (0,1)\;\text{and}\;\ell \in \mathbb{N} \;\text{such that}\;\lvert f^{n+\ell}(z)\rvert \geq \mu_{2,\varepsilon}^n(R),\;\text{for}\; n \in \mathbb{N} \},
 \end{equation}
 where $R>0$ is such that $\mu_{2,\varepsilon}(r)>r$ for $r\geq R.$ Note that $Q_2(f)$ is independent of~$R$.\\
 
 For $0<\varepsilon<1$ we have $\mu_{2,\varepsilon}(r)< \mu_{\varepsilon}(r),$ for sufficiently large $r$, so
 $$A(f)\subset Q(f) \subset Q_2(f).$$
 
Unlike the functions $\mu_{\varepsilon}$ that were introduced in earlier papers, for $\mu_{2,\varepsilon}$ we do  not know a priori that, for any given transcendental entire function, $\mu_{2,\varepsilon}(r)>r$ for $r$ large enough. This means that, for some slowly growing functions $f$, there exist points in $Q_2(f)$ that are not even escaping. However, $Q_2(f) \subset I(f)$ for a large class of functions. We seek to identify functions for which  $Q_2(f)=A(f).$\\

Recall that the order $\rho (f)$ and lower order $\lambda (f)$ of $f$ are defined by
$$\rho (f)=\limsup_{r \to \infty} \frac{\log \log M(r)}{\log r},\;\;\; \lambda(f)= \liminf_{r \to \infty} \frac{\log \log M(r)}{\log r}.$$
We prove the following result:
\begin{theorem}
\label{thintro}
Let $f= f_1 \circ f_2 \circ \cdots \circ f_j$ be a finite composition of transcendental entire functions, where $f_1$ has finite order and positive lower order. Then $Q_2(f)=~A(f)$.
\end{theorem}
An immediate consequence of Theorem \ref{thintro} is the following:
\begin{cor}
\label{corintro}
We have $Q_2(f)=A(f)$ whenever $f= f_1 \circ f_2 \circ \cdots \circ f_j$ is a finite composition of transcendental entire functions and $f_1$  satisfies one of the following:\\
\begin{itemize}
\item[(a)] there exist $A, B, C, r_0>1$ such that
$$A\log M(r,f_1) \leq \log M(Cr,f_1) \leq B\log M(r,f_1),\;\;\text{for}\;\;r \geq r_0;$$
\item[(b)] $f_1 \in \mathcal{B}$ and is of finite order.\\
\end{itemize}
\end{cor}

 In fact, functions of type (a) were studied by Bergweiler and Karpi\'{n}ska in \cite{Haus} where it was shown that they are of finite order and positive lower order. All functions in class $\mathcal{B}$ have lower order not less than $1/2$ (see \cite[Lemma 3.5]{Dimen}) and finite compositions of such functions of finite order were considered by  Rottenfusser, R\"uckert, Rempe and Schleicher in \cite{3R}.\\
 
The proof of Theorem \ref{thintro} is in three steps. We first introduce a new regularity condition as follows:\\

A transcendental entire function $f$ is \textit{strongly log-regular} if, for any $\varepsilon \in (0,1)$, there exist $R>0$ and $k>1$ such that 
\begin{equation}
\label{nrc}
\log M(r^k) \geq (k\log M(r))^{1/\varepsilon},\;\;\text{for}\;\;r>R.
\end{equation} 

Using (\ref{defmu}) we see that for each $\varepsilon \in (0,1),$  (\ref{nrc}) is equivalent to 
\begin{equation}
\label{slogm}
\mu_{2,\varepsilon}(r^k) \geq M(r)^k,\;\;\text{for}\;\;r>R,
\end{equation}
which implies that
\begin{equation}
\label{logm}
\mu_{1,\varepsilon}(r^k) \geq M(r)^k,\;\;\text{for large}\;\;r,
\end{equation}
or equivalently, there exist $R_0>0,  k, d >1$ such that
\begin{equation}
\label{logreg2}
M(r^k) \geq M(r)^{kd},\;\;\text{for}\;\;r >R_0.
\end{equation}
 The latter condition is equivalent to the condition called \textit{log-regularity} that was used in \cite{Regul} as a sufficient condition for $Q(f)=A(f).$ The name strong log-regularity arises from the fact that strong log-regularity implies log-regularity.\\
 
It seems natural to generalise (\ref{q2def}) and  (\ref{slogm}) for any $m \in \mathbb{N}$ as follows:\\
Let
$$Q_m(f)= \{z: \exists\;\varepsilon \in (0,1), \ell \in \mathbb{N}\;\; \text{such that}\;\; \lvert f^{n+ \ell}(z) \rvert \geq \mu_{m,\varepsilon}^n(R),\; \text{for}\;\; n\in \mathbb{N}\},$$
where $R>0$ is such that $\mu_{m,\varepsilon}(r)> r$ for $r \geq R$. If $\mu_{m,\varepsilon}$ satisfies the generalised (\ref{slogm}), that is, for any $\varepsilon \in (0,1)$ and any $m \in \mathbb{N},$ there exist $R>0$ and $k>1$ such that
\begin{equation}
\label{slrgen}
\mu_{m,\varepsilon}(r^k) \geq M(r)^k,\;\;\text{for}\;\;r>R,
\end{equation} 
then we can show that $Q_m(f)=A(f)$. However, the larger $m$ is, the more difficult it is for $f$ to satisfy (\ref{slrgen}). For example, it is not hard to check that, for $m=3$, $f(z)=e^z$ does not satisfy (\ref{slrgen}). In forthcoming work we give alternative, but more complicated, regularity conditions which, for any $m \geq 2$, guarantee that $Q_m(f)=A(f)$ for a wide range of functions $f$. \\

In Section 2, we give the first two steps of the proof of Theorem \ref{thintro}.
In the first step, we show that strong log-regularity is a sufficient condition for $Q_2(f)=A(f).$ In the second step, we prove that any transcendental entire function of finite order and positive lower order is strongly log-regular.\\

The last step of the proof is given in Section 3 where we show that strong log-regularity is preserved under finite composition of transcendental entire functions where the first function of the composition is strongly log-regular.\\

Finally, in the last section, we construct two functions. The first is  an example of a strongly log-regular function with zero lower order and positive, finite order and the second is a function for which  $Q_2(f) \neq A(f)$ whereas $Q(f)=A(f).$\\

\textit{Acknowledgment.} I would like to thank my supervisors Prof. Phil Rippon and Prof. Gwyneth Stallard for their patient guidance and all their help with this paper.

\section{Sufficient conditions for $Q_2(f)=A(f)$}
\label{strong-log}

In this section we give the first two steps of the proof of Theorem \ref{thintro}. 
We first prove the following result about strongly log-regular functions:
\begin{theorem}
\label{strongQA}
Let $f$ be a transcendental entire function which is strongly log-regular. Then $Q_2(f)=A(f).$
\end{theorem}
\begin{proof}
Clearly $A(f) \subset Q_2(f)$, as noted earlier. Suppose now that $z \in Q_2(f).$ Then (\ref{q2def}) implies that there exist $\varepsilon \in (0,1)$ and $\ell \in \mathbb{N}$ such that  
\begin{equation}
\label{q2}
\lvert f^{n+\ell}(z) \rvert \geq \mu_{2,\varepsilon}^n(R),\;\;\text{for}\;\;n \in \mathbb{N},
\end{equation}
where $R>0$ is such that $\mu_{2,\varepsilon}(r)> r$ for $r \geq R$.
As $f$ is strongly log-regular it satisfies (\ref{slogm}) and so there exist $R_0>R$ and $k>1$ such that, for $r>R_0,$
 \begin{equation}
\label{eth}
\mu_{2,\varepsilon}(r^k) \geq  M(r)^k,\;\;\text{for}\;\;r>R_0.
\end{equation}
By applying (\ref{eth}) twice we obtain $$\mu_{2,\varepsilon}^2(r^k)= \mu_{2,\varepsilon}(\mu_{2,\varepsilon}(r^k)) \geq \mu_{2,\varepsilon} ((M(r))^k) \geq (M(M(r)))^k,\;\;\text{for}\;\;r>R_0$$ since $\mu_{2,\varepsilon}(r^k)>R.$ By applying (\ref{eth}) repeatedly in this way we obtain that
\begin{equation}
\label{q2.2}
\mu_{2,\varepsilon}^{n}(r^k) \geq (M^{n}(r))^k \geq M^{n}(r),\;\;\text{for}\;\;r>R_0.
\end{equation}

But $M^n(r) \to \infty$ as $n \to \infty$ for $r \geq R$ and so there exists $n_0 \in \mathbb{N}$ such that $M^{n_0}(R) \geq R^k$ and hence, (\ref{q2}) and (\ref{q2.2}) imply that
$$\lvert f^{n+ n_0+\ell}(z) \rvert \geq \mu_{2,\varepsilon}^{n+n_0}(R) \geq M^{n+n_0}(R^{1/k}) \geq M^n(R),$$
and the result follows.
\end{proof}

We now show that all functions of finite order and positive lower order are strongly log-regular.

\begin{theorem}
\label{th3}
\label{mt}
Let $f$ be a transcendental entire function of finite order and positive lower order. Then $f$ is strongly log-regular and hence $Q_2(f)=A(f).$
\end{theorem}
\begin{proof}
Let $f$ be a transcendental entire function of finite order and positive lower order. Then there exist $0<q<p$ such that
\begin{equation}
\label{th3.1}
e^{r^q} \leq M(r) \leq e^{r^p},\;\text{for sufficiently large}\;r,
\end{equation}
or equivalently
$$r^q \leq \log M(r) \leq r^p.$$
So, for each $\varepsilon \in (0,1)$ and sufficiently large $r$,
\begin{equation}
\label{th3.2}
(\log M(r^k))^{\varepsilon} \geq (r^{qk})^{\varepsilon}= r^{\varepsilon qk}.
 \end{equation}
It follows from (\ref{th3.1}) that, for $k> p/(q\varepsilon),$ there exists $R>0$ such that, for $r>R$,
\begin{equation}
\label{th3.3}
r^{\varepsilon qk} \geq kr^p \geq k\log M(r), 
\end{equation}
so (\ref{nrc}) is satisfied, by (\ref{th3.2}) and (\ref{th3.3}).
\end{proof}

\section{Composition and strong log-regularity}
\label{compose}
In this section we complete the proof of Theorem \ref{thintro} by showing that the finite composition of transcendental entire functions, where the first function of the composition is strongly log-regular, is a strongly log-regular function.

\begin{theorem}
\label{thcom}
Let  $f_1,f_2,...,f_j$ be transcendental entire functions and suppose $f_1$ is strongly log-regular. Then $g= f_1 \circ f_2 \circ...\circ f_j$ is strongly log-regular.
\end{theorem}
Theorem \ref{thcom} implies that if $f$ is strongly log-regular then the $n$-th iterate $f^n$ is  strongly log-regular as well.\\

In order to prove the theorem we need the following lemma of Rippon and Stallard \cite[Lemma 2.2]{Smallg}.\\ 

\begin{lemma}

Let $f$ be a transcendental entire function. Then there exists $R_0>0$ such that, for all $r \geq R_0$ and all $c>1$,
\begin{equation}
\label{max-prop}
\log M(r^c) \geq c \log M(r).
\end{equation}
\end{lemma}
We also need the following lemma of Sixsmith \cite[Lemma 2.4]{DSix}.
\begin{lemma}
\label{lemdave}
Suppose that $f$ is a non-constant entire function and $g$ is a transcendental entire function. Then, given $\nu >1$, there exist $R_1, R_2 >0$ such that
\begin{equation}
\label{Dave1}
M(\nu r, f\circ g) \geq M(M(r,g),f) \geq M(r,f\circ g),\;\;\text{for}\;\;r\geq R_1
\end{equation}
and
\begin{equation}
\label{Dave2}
M(\nu r, g\circ f) \geq M(M(r,f),g) \geq M(r,g\circ f),\;\;\text{for}\;\;r\geq R_2.
\end{equation}
\end{lemma}

\begin{proof}[Proof of Theorem \ref{thcom}]
It is sufficient to prove the result for $j=2.$ Let $f_1$ be strongly log-regular, that is, for any $\varepsilon \in (0,1)$ there exist $R>0$ and $k>1$ such that
\begin{equation}
\label{f1}
(\log M(r^k,f_1))^{\varepsilon} \geq k\log M(r,f_1),\;\;\text{for}\;\;r\geq R, 
\end{equation}
and let $f_2$ be any transcendental entire function.\\

Given $\varepsilon' \in (0,1)$ we take $\varepsilon= \frac{2}{3} \varepsilon'.$ Then there exist $R>0$ and $k>1$ such that (\ref{f1}) holds with this $\varepsilon.$  Now take  $\nu =k^{1/2}$ and put $k'= \nu k= k^{3/2}$. Note that $\varepsilon'= \frac{3}{2} \varepsilon= \varepsilon(1+ \log \nu/\log k).$  Then we apply Lemma \ref{lemdave} with $f=f_2$ and $g=f_1,$ where $R_2$ is the constant in (\ref{Dave2}) and $R_0$ is the constant in (\ref{max-prop}) for $f=f_2.$ So, for $r \geq \max\{e,R_0, R_2\}$, we have 
\begin{eqnarray}
M(r^{k'},f_1\circ f_2) &\geq & M(\nu r^k, f_1\circ f_2) \nonumber \\
&\geq & M(M(r^k,f_2),f_1), \;\;\;\;\;\;\;\text{by (\ref{Dave2})} \nonumber \\
& \geq & M(M(r,f_2)^k,f_1),\;\;\;\;\;\;\;\text{by (\ref{max-prop})}.\nonumber
\end{eqnarray}

Hence, for $r \geq R'= \max\{e, R, R_0, R_2\},$
\begin{eqnarray}
(\log M(r^{k'},f_1\circ f_2))^{\varepsilon} &\geq & (\log M(M(r,f_2)^k,f_1))^{\varepsilon} \nonumber \\
&\geq & k\log M(M(r,f_2),f_1), \;\;\;\text{by (\ref{f1})} \nonumber \\
&\geq& k\log M(r,f_1\circ f_2).\nonumber 
\end{eqnarray}
Hence, for $r \geq R',$
\begin{eqnarray}
(\log M(r^{k'},f_1\circ f_2))^{\varepsilon'}&=& (\log M(r^{k'},f_1\circ f_2))^{(3/2)\varepsilon} \nonumber \\
& \geq &(k\log M(r,f_1\circ f_2))^{3/2} \nonumber \\
&=& k' (\log M(r,f_1\circ f_2))^{3/2} \nonumber \\
&\geq &  k' \log M(r,f_1\circ f_2), \nonumber
\end{eqnarray}
as required.
So $f_1 \circ f_2$ is strongly log-regular.
\end{proof}

 \section{Examples}
 In this section we construct two examples of functions with specific properties.\\
 
\begin{ex}
\label{th}

There exists a transcendental entire function of zero lower order and positive, finite order which is strongly log-regular.
\end{ex}

\begin{ex}
\label{ex2}
There exists a transcendental entire function $f$ which is log-regular such that $Q_2(f) \neq A(f).$ Hence, $Q(f)=A(f)$ but $Q_2(f) \neq Q(f).$
\end{ex} 
In order to construct these functions we use the following lemma (see \cite{Clunie}).\\
 \begin{lemma}
\label{C-K}
Let $\phi$ be a convex increasing function on $\mathbb{R}$ such that $\phi(t) \neq O(t)$ as $t \to \infty$. Then there exists a transcendental entire function $f$ such that 
$$\log M(e^t,f) \sim  \phi(t)\;\;\text{as}\;\;t \to \infty.$$
\end{lemma}

 We showed in Section \ref{strong-log} that all transcendental entire functions of finite order and positive lower order are strongly log-regular. However, a strongly log-regular function of finite order does not need to have positive lower order. Indeed, Example \ref{th} gives a function of zero lower order and positive, finite order which is strongly log-regular.\\

\begin{proof}[Proof of Example \ref{th}] We first take a fixed value of $\varepsilon$, say $\tilde{\varepsilon} \in (0,1),$ and a fixed value of $k$, say $\tilde{k}$, such that $\tilde{k}>2/\tilde{\varepsilon} \geq \frac{2\log (\tilde{k}+1)}{\log \tilde{k}}$ and construct a convex increasing function $\phi$ on $\mathbb{R}$ such that:\\
\begin{itemize}
 \item[a)] $\displaystyle  \liminf_{t \to \infty} \frac{\log \phi (t)}{t}= 0;$

\item[b)] $\displaystyle 1 \leq \limsup_{t \to \infty} \frac{\log \phi (t)}{t} \leq \tilde{k};$

\item[c)] there exists $T>0$ such that, for $t>T,$
\begin{equation}
\label{1.1}
\phi (\tilde{k}t) \geq (\tilde{k} \phi (t))^{1/\tilde{\varepsilon}}. 
\end{equation}
\end{itemize}

Once this is done, we show that this function $\phi$ satisfies (\ref{1.1}) for \textit{any} $\varepsilon \in(0,1)$ with a suitable $k>1$.\\

Take $\tilde{d}= 1/\tilde{\varepsilon}.$ Then $\tilde{k}>\tilde{d}>1.$ Take $a_0=1$, and choose $t_0$ so large that 
\begin{equation}
\label{deft0}
\frac{\log \tilde{k}}{t_0}< \frac{1}{2}
\end{equation}
and
\begin{equation}
\label{deft02}
\tilde{k}^{\tilde{d}}e^{\tilde{d}t} \leq e^{\tilde{k}t},\;\;\text{for}\;\;t \geq t_0.
\end{equation}

 Then we set $t_n= \tilde{k}^nt_0,$ for $n \in \mathbb{N},$ and define 
\begin{equation}
\label{an1}
a_n= e^{t_n},\;\;n=N_1,N_2,...,N_m,...,
\end{equation}
 where $(N_m)$ is an increasing sequence, to be chosen shortly, and
 \begin{equation}
 \label{an2}
a_n= (\tilde{k}a_{n-1})^{\tilde{d}},\;\;\text{elsewhere}.
\end{equation}

 We will show that,  for each $m \in \mathbb{N},$ we can choose $N_m$ so that
\begin{equation}
\label{Nm}
 \frac{\log a_{N_m-1}}{t_{N_m-1}}< \frac{1}{2^m}
 \end{equation}
 and
 \begin{equation}
 \label{3-2}
 e^{t_{N_m}} \geq (\tilde{k} a_{N_m-1})^{\tilde{d}}.
 \end{equation}

Then we let $\phi$ be the real function that is linear on each of the intervals $[t_n,t_{n+1}]$ with $\phi (t_n)= a_n$, for $n \in \mathbb{N}.$\\
 
Suppose there is no $N_1 \in \mathbb{N}$ which satisfies (\ref{Nm}) with $m=1$. Then $a_n= (\tilde{k}a_{n-1})^{\tilde{d}}$ for all $n \in \mathbb{N}.$ Hence,
 \begin{eqnarray}
  \frac{\log a_n}{t_n} &=& \frac{\log (\tilde{k}a_{n-1})^{\tilde{d}}}{t_n}\nonumber \\
  &=& \frac{\tilde{d}}{\tilde{k}} \left(\frac{\log \tilde{k}}{t_{n-1}}+ \frac{\log a_{n-1}}{t_{n-1}}\right). 
 \end{eqnarray}
Now let $x_n= \frac{\log a_n}{t_n}, c= \tilde{d}/\tilde{k} <1$ and $\varepsilon_{n-1}= \frac{\log \tilde{k}}{t_{n-1}}$. We have that
$$x_n= c( \varepsilon_{n-1}+ x_{n-1}),\;\;\text{for all}\;\;n \in \mathbb{N},$$ so
\begin{eqnarray}
\limsup_{n \to \infty} x_n &\leq & c\limsup_{n \to \infty} \varepsilon_{n-1}+ c\limsup_{n \to \infty} x_{n-1} \nonumber \\
&= & c\limsup_{n \to \infty} x_{n-1}, \nonumber
\end{eqnarray}
as $\varepsilon_{n-1} \to 0$ as $n \to \infty$ and $c<1.$ Hence $$\limsup_{n \to \infty}  \frac{\log a_n}{t_n}=0$$ and so we obtain a contradiction. Therefore, (\ref{Nm}) is true for some $N_1 \in \mathbb{N}.$\\

Suppose now that (\ref{Nm}) is true for $N_1,N_2,...,N_m \in \mathbb{N}$ but it fails to be true for all $n> N_m.$ Following the above argument, we again obtain a contradiction and so there exists $N_{m+1} \in \mathbb{N}$ which satisfies (\ref{Nm}).\\

It follows from (\ref{deft0}),(\ref{Nm}) and the fact that $\tilde{k}>2\tilde{d}>1$ that
$$\frac{\log (\tilde{k}a_{N_m-1})^{\tilde{d}}}{t_{N_m}}= \frac{\tilde{d}}{\tilde{k}} \left(\frac{\log \tilde{k}}{t_{N_m-1}}+ \frac{\log a_{N_m-1}}{t_{N_m-1}}\right) \leq \frac{\log a_{N_m-1}}{t_{N_m-1}} < \frac{1}{2^m}<1,$$
and so $e^{t_{N_m}} > (\tilde{k}a_{N_m-1})^{\tilde{d}}$, which means that (\ref{Nm}) implies (\ref{3-2}).\\

 In order to prove a) we note that it follows from (\ref{Nm}) that $$\frac{\log \phi(t_{N_m-1})}{t_{N_m-1}}= \frac{\log a_{N_m-1}}{t_{N_m-1}}< \frac{1}{2^m},\;\;\text{for}\;\;m \in \mathbb{N},$$
and so $$\liminf_{t \to \infty} \frac{\log \phi(t)}{t} \leq \liminf_{m \to \infty} \frac{1}{2^m} =0.$$

We also note that it follows from (\ref{an1}) that $$\frac{\log \phi(t_{N_m})}{t_{N_m}}= \frac{\log a_{N_m}}{t_{N_m}}=1,\;\;\text{for}\;\;m \in \mathbb{N},$$ and so, in order to prove b), it remains to show that $$\limsup_{t \to \infty} \frac{\log \phi (t)}{t} \leq \tilde{k}.$$
It suffices to show that $\phi(t) \leq e^{\tilde{k}t}$ for large values of $t$. We will first show that $\phi(t_n) \leq e^{t_n},$ for $n$ large enough.\\

Suppose that $\phi(t_n) \leq e^{t_n}$ for some $n.$ Then either 
$$\phi(t_{n+1}) = e^{t_{n+1}}$$
or 
\begin{equation}
\label{neworder}
\phi(t_{n+1}) = (\tilde{k}\phi(t_n))^{\tilde{d}} \leq (\tilde{k}e^{t_n})^{\tilde{d}}
 \leq  e^{\tilde{k}t_n}=e^{t_{n+1}},
\end{equation}
by (\ref{deft02}). In either case, we deduce that $\phi(t_n) \leq e^{t_n}$ implies that $\phi(t_{n+1}) \leq e^{t_{n+1}}$. Since $\phi(t_{N_m})= e^{t_{N_m}}$, for all $m \in \mathbb{N},$  we conclude that $\phi(t_n) \leq e^{t_n}$, for $t_n \geq t_{N_1}.$ Now take any $t \in [t_n,t_{n+1}], n \geq N_1.$ Then

$$\phi(t) \leq \phi(t_{n+1}) \leq e^{t_{n+1}}= e^{\tilde{k}t_n} \leq e^{\tilde{k}t},$$
and the result follows.\\

We now show that (\ref{1.1}) is true for $t \geq t_0.$ In order to do so, we  consider the functions $g(t)= \phi (\tilde{k}t)$ and $h(t)= (\tilde{k} \phi(t))^{\tilde{d}}$. For each $n\geq 0$, $g$ is a linear, increasing function on $[t_n,t_{n+1}]= [t_n, \tilde{k} t_n]$ and $h$ is convex on $[t_n,t_{n+1}]$. We will find the values of the two functions $g$ and $h$ at the endpoints of each interval and we will use the fact that the graph of a convex function which has the same or smaller values at the endpoints than a linear function is always below the graph of the linear function. Thus, to show that $h(t)= (\tilde{k} \phi(t))^{\tilde{d}} \leq \phi(\tilde{k} t)= g(t)$ for all $t \geq t_0,$ it is sufficient to show that
$$(\tilde{k} \phi(t_n))^{\tilde{d}} \leq \phi(t_{n+1}),\;\;\text{for}\;\;n \geq 0.$$
This is evidently true if (\ref{an2}) holds and if (\ref{an1}) holds it is true by (\ref{3-2}).\\

Finally, we need to show that $\phi$ is convex. It suffices to show that the sequence of gradients $g_n= \frac{a_n-a_{n-1}}{t_n-t_{n-1}}$, $n \in \mathbb{N},$ of the line segments in the graph of $\phi$ is increasing, or equivalently that, for $n \in \mathbb{N}$,
\begin{equation}
\label{convex}
\frac{a_{n+1}-a_n}{t_{n+1}-t_n} \geq \frac{a_n-a_{n-1}}{t_n-t_{n-1}}.
\end{equation}
Since $t_n=\tilde{k}^nt_0,$ for $n \in \mathbb{N},$ we need to show that
$$a_{n+1} \geq (\tilde{k}+1)a_n -\tilde{k}a_{n-1}.$$ 
But $$a_{n+1}+ \tilde{k}a_{n-1} \geq (\tilde{k}+1)a_n$$ since, by (\ref{3-2}) and the fact that $\tilde{d}= \frac{1}{\tilde{\varepsilon}} \geq \frac{\log (\tilde{k}+1)}{\log \tilde{k}}$,
$$a_{n+1} \geq (\tilde{k}a_n)^{\tilde{d}} \geq \tilde{k}^{\tilde{d}} a_n \geq (\tilde{k}+1) a_n,\;\;\text{for}\;\;n \in \mathbb{N},$$ and the result follows.\\

We have constructed a function $\phi$ such that (\ref{1.1}) holds for $\tilde{\varepsilon}$ which is a specific value of $\varepsilon \in (0,1)$. In fact for any other $\varepsilon \in (0,\tilde{\varepsilon})$ we can find a  large enough $k>1$ such that (\ref{1.1}) holds for the same function $\phi.$ Indeed, suppose first that (\ref{1.1}) holds for $\varepsilon= \tilde{\varepsilon}$ and set $\tilde{d}= 1/\tilde{\varepsilon},$ as before.\\

Now take $\varepsilon \in (0, \tilde{\varepsilon})$ and suppose that $1/\varepsilon= d= {\tilde{d}}^p,$ for some $n\leq p<n+1, n \in \mathbb{N}$. It follows from (\ref{1.1}) that, for $t \geq t_0,$ 
\begin{eqnarray}
\phi(\tilde{k}^{2n+2}t) &\geq & (\tilde{k}\phi(\tilde{k}^{2n+1}t))^{\tilde{d}} = \tilde{k}^{\tilde{d}} \phi({\tilde{k}}^{2n+1}t)^{\tilde{d}} \nonumber \\
\label{lasteq}
&\geq & {\tilde{k}}^{\tilde{d}} \tilde{k}^{{\tilde{d}}^2} \phi (\tilde{k}^{2n}t)^{\tilde{d}^2} \nonumber \\
& \geq & \tilde{k}^{\tilde{d}+{\tilde{d}}^2+ \cdots +{\tilde{d}}^{2n+2}} \phi(t)^{\tilde{d}^{2n+2}}.
\end{eqnarray}

We now show that 
\begin{equation}
\label{indep}
\tilde{k}^{\tilde{d}+{\tilde{d}}^2+\cdots +{\tilde{d}}^{2n+2}} \phi(t)^{{\tilde{d}}^{2n+2}} \geq ({\tilde{k}}^{2n+2} \phi(t))^{\tilde{d}^p}.
\end{equation}
As $p<n+1$, it suffices to show that
\begin{equation}
\label{indep2}
\tilde{d}+\tilde{d}^2+\cdots +\tilde{d}^{2n+2} \geq (2n+2){\tilde{d}}^{n+1}.
\end{equation}\\

We will prove (\ref{indep2}) using the inequality of arithmetic and geometric means, which implies that
\begin{eqnarray}
\frac{\tilde{d}+{\tilde{d}}^2+ \cdots +{\tilde{d}}^{2n+2}}{2n+2} &\geq& \sqrt[2n+2]{\tilde{d}  {\tilde{d}}^2\cdots {\tilde{d}}^{2n+2}} \nonumber \\
&=& \sqrt[2n+2]{{\tilde{d}}^{(2n+2)(2n+3)/2}} \nonumber \\
&=& \tilde{d}^{(2n+3)/2} ={\tilde{d}}^{n+ 3/2} > \tilde{d}^{n+1}, \nonumber
\end{eqnarray}

 as required. Combining (\ref{lasteq}) and (\ref{indep}) gives 
 $$\phi(kt) \geq (k\phi(t))^{\varepsilon},\;\;\text{for}\;\;t \geq t_0,$$
 where $k= \tilde{k}^{2n+2}.$ Thus (\ref{1.1}) holds for any $\varepsilon \in (0, \tilde{\varepsilon})$ and hence for any $\varepsilon \in (0,1)$.\\

Now we can apply Lemma \ref{C-K} to $\phi$ to give a transcendental entire function $f$ such that
\begin{equation}
\label{tef}
\log M(e^t,f)= \phi(t)(1+ \delta(t)),
\end{equation}
 where $\delta(t) \to 0$ as $t \to \infty.$ Then
 $$\frac{\log \log M(e^t,f)}{t}= \frac{\log \phi(t)}{t}+ \frac{O(\delta(t))}{t}$$
 and so $\lambda(f)=0$ and $1 \leq \rho(f) \leq \tilde{k}$, by properties (a) and (b) respectively.

 It remains to show that $f$ satisfies (\ref{nrc}). We know that  for any $\varepsilon \in (0,1)$ there exists $k> d= 1/\varepsilon$ such that
 \begin{equation}
 \label{slogex} 
 \phi (kt) \geq (k\phi(t))^d,\;\;\text{for}\;\;t \geq t_0.
 \end{equation}
 Let $0<\varepsilon'< \varepsilon$ and set $d'= 1/\varepsilon'.$ Then, by (\ref{tef}) and (\ref{slogex}), 
$$\log M(e^{kt},f) \geq \frac{1+ \delta(kt)}{(1+ \delta(t))^{d'}} (k\log M(e^t,f))^{d'},\;\;\text{for}\;\;t \geq t_0.$$
Since $$\frac{1+ \delta(kt)}{(1+ \delta(t))^{d'}} \to 1\;\;\text{as}\;\;t \to \infty,$$
we deduce that 
$$\log M(e^{kt},f) \geq (k\log M(e^t,f))^{1/\varepsilon},$$
for large $t$ and so $f$ satisfies (\ref{nrc}) for sufficiently large $r$ with $r= e^t$. 
\end{proof}

Throughout the paper, we are interested in sufficient conditions for $Q_2(f)=~ A(f)$. However, these two sets are not always equal. We now construct a function for which $Q_2(f)$ is not equal to $A(f)$. \\

\begin{proof}[Proof of Example \ref{ex2}] We construct a transcendental entire function $f$ which is log-regular and hence, by \cite[Theorem 4.1]{Regul}, $Q(f)=A(f),$ but for which $Q_2(f)\neq~ A(f).$ Obviously, this function cannot be strongly log-regular. In order to construct such a function we will again use the result of Clunie and K\"{o}vari (see Lemma \ref{C-K}). The idea is to find a real, increasing, convex function $\phi$  such that:\\
\begin{itemize}
\item[$\bullet$] there exist $ k>1$ and $d>1$ such that
\begin{equation}
\label{sq1.1}
\phi(kt) \geq kd \phi(t),\;\;\text{for large}\;\;t,
\end{equation}
and\\

\item[$\bullet$] if $f$ is produced from $\phi$ using Lemma \ref{C-K} then the iterates of the function $\mu_{2,\varepsilon}(r),$ for $\varepsilon \in (1/2,1)$, grow much more slowly than the iterates of $M(r)= M(r,f).$\\
\end{itemize}
Let $\phi(t)= t^2, t>0.$ Then $\phi$ is increasing and convex. Let $k>1$ and $1<d<k.$ Then 
$$\phi(kt)= k^2t^2 > kd t^2  = kd \phi(t)$$ and so (\ref{sq1.1}) is satisfied.\\

Now we apply Lemma \ref{C-K} to $\phi$ to give a transcendental entire function $f$ such that
\begin{equation}
\label{sq1.2}
\log M(e^t,f)= \phi(t)(1+ \delta(t))= t^2(1+ \delta(t)),
\end{equation}
 where $\delta(t) \to 0$ as $t \to \infty.$
 
 Now (\ref{sq1.1}) implies that 
$$\log M(e^{kt},f) \geq kd \frac{1+ \delta(kt)}{1+ \delta(t)} \log M(e^t,f),$$
where $$\frac{1+ \delta(kt)}{1+ \delta(t)} \to 1\;\; \text{as}\;\; t \to \infty.$$
Hence, there exists $1<d'<d$  such that 
$$\log M(e^{kt},f) \geq kd' \log M(e^t,f),$$
for large $t$, and so, by (\ref{logreg2}), $f$ is log-regular  which implies that $Q(f)=A(f).$\\

Now we show that, for $\varepsilon \in (1/2,1)$, the iterates of $\mu_{2,\varepsilon}(r)$ grow  more slowly than the iterates of $M(r).$\\

By (\ref{sq1.2}), we have 
\begin{equation}
\label{2ex}
M(r)=\exp ((\log r)^2(1+ \nu(r))),
\end{equation}
where $\nu(r)= \delta (\log r) \to 0,$ as $r \to \infty,$ and 
\begin{equation}
\label{ex2.2}
\mu_{2,\varepsilon}(r)= \exp ((\log M(r))^{\varepsilon})= \exp (((\log r)^2(1+ \nu(r)))^{\varepsilon})= \exp ((\log r)^{2\varepsilon}(1+ \nu(r))^{\varepsilon}).
\end{equation}
 
Now fix $\varepsilon \in (1/2,1).$ It then follows from (\ref{ex2.2}) that there exists $R>0$ such that we have $\mu_{2,\varepsilon}(r)>r,$ for $r\geq R$ and so $\mu_{2,\varepsilon}^n(R) >R$, for $n \in \mathbb{N}.$ \\

 The idea is to show that, for any $m \in \mathbb{N},$ there exists $N \in \mathbb{N}$ such that, for any $n \in \mathbb{N}$ with $n>N$, we have
 \begin{equation}
 \label{sq1.3}
 \mu_{2,\varepsilon}^{m+n}(R_0)< M^n(R_0),
 \end{equation}
for some $R_0\geq R.$ We then show that this implies that $Q_2(f) \neq A(f).$
 
Since $\varepsilon \in (1/2,1)$, it follows from  (\ref{2ex}) and (\ref{ex2.2}) that there exist $R_1, R_2>0$ and $c, \tilde{c} \in \mathbb{R}$ such that 
\begin{equation}
\label{sq1.11}
1<2\varepsilon <\tilde{c}<c<2,
\end{equation}
\begin{equation}
\label{sq1.12}
M(r) \geq \exp((\log r)^c),\;\;\text{for}\;\;r>R_1,
\end{equation}
and
\begin{equation}
\label{sq1.13}
\mu_{2,\varepsilon}(r) \leq \exp ((\log r)^{\tilde{c}}),\;\;\text{for}\;\;r>R_2.
\end{equation}
 Hence, by (\ref{sq1.12}) and (\ref{sq1.13}), we obtain, for $n \in \mathbb{N},$
$$M^n(r) \geq \exp((\log r)^{c^n}),\;\;\text{for}\;\;r>R_1,$$ and
$$\mu_{2,\varepsilon}^n(r) \leq \exp((\log r)^{{\tilde{c}}^n}),\;\;\text{for}\;\;r>R_2.$$
By (\ref{sq1.11}), we can easily see that, for any $m \in \mathbb{N},$ there exists $N \in \mathbb{N}$ such that, for any $n \in \mathbb{N}$ with $n>N$, we have
$${\tilde{c}}^{n+m} < c^n,$$
and hence
$$\mu_{2,\varepsilon}^{n+m}(r) \leq \exp((\log r)^{{\tilde{c}}^{n+m}})<\exp((\log r)^{c^n}) \leq M^n(r),\;\;\text{for}\;\;r>R_0,$$
where $R_0= \max\{R, R_1, R_2\}.$ Therefore, (\ref{sq1.3}) is satisfied for $R_0= \max\{R, R_1, R_2\}.$
We will now show that (\ref{sq1.3}) implies that $Q_2(f)\setminus A(f)$ is non-empty. For this purpose we will use the following theorem of Rippon and Stallard (see \cite[Theorem~ 3.1]{Regul}).
\begin{theorem}
\label{rs}
Let $f$ be a transcendental entire function. There exists $R=~ R(f)>~0$ with the property that whenever $(a_n)$ is a positive sequence such that 
$$a_n \geq R \;\;\text{and}\;\;a_{n+1} \leq M(a_n),\;\;\text{for}\;\; n\in \mathbb{N},$$

there exists a point $\zeta \in \mathbb{C}$ and a sequence $(n_j)$ with $n_j \to \infty$ such that 
$$\lvert f^n(\zeta) \rvert \geq a_n, \;\;\text{for}\;\;n \in \mathbb{N},\;\;\text{but}\;\;\lvert f^{n_j}(\zeta) \rvert \leq M^2(a_{n_j}),\;\;\text{for}\;\;j \in \mathbb{N}.$$

\end{theorem}

Now, by Theorem \ref{rs}, with $a_n= \mu_{2,\varepsilon}^n(R), n \in \mathbb{N},$ there exists a point $\zeta$ and a sequence $(n_j)$ with $(n_j) \to \infty$ as $j \to \infty $, such that, for our function $f$,
\begin{equation}
\label{sq1.9}
\lvert f^n(\zeta) \rvert \geq \mu_{2,\varepsilon}^n(R),\;\;\text{for}\;\;n \in \mathbb{N},
\end{equation}
and
\begin{equation}
\label{sq1.10}
\lvert f^{n_j}(\zeta) \rvert \leq M^2(\mu_{2,\varepsilon}^{n_j}(R)), \;\;\text{for}\;\;j \in \mathbb{N}.
\end{equation}
It follows from (\ref{sq1.9}) that $\zeta \in Q_2(f).$ Also, (\ref{sq1.3}) and (\ref{sq1.10}) together imply that, for each $m \in \mathbb{N}$ and $n_j>m,$ we have
\begin{eqnarray}
\lvert f^{(n_j-m +2)+m -2}(\zeta) \rvert &=& \lvert f^{n_j}(\zeta) \rvert \nonumber \\
& \leq & M^2(\mu_{2,\varepsilon}^{n_j}(R))\nonumber \\
&<& M^2(M^{n_j-m}(R))\nonumber \\
&=&  M^{n_j-m+2}(R).\nonumber
\end{eqnarray}
 Hence, $\zeta \notin A(f),$ so $Q_2(f) \neq A(f)$, as required.\end{proof}

\end{document}